\documentclass[reqno]{amsart}
\usepackage{amscd,amssymb,amsmath,amsfonts,amsthm,latexsym}
\usepackage{extpfeil}
\usepackage[latin1]{inputenc}
\usepackage[all]{xy}

\newcommand{\A}{\mathrm A}

\def\t{\otimes}
\def \Im{\mathop{\sf Im}\nolimits}

\DeclareMathOperator{\Ker}{\sf Ker}

\def \Hom{\mathop{\sf Hom}\nolimits}

\def \Zenter{\mathop{\sf Z}\nolimits}
\def \universal{\mathop{\sf U}\nolimits}

\newcommand{\conm}[1]{\{#1, #1\}}

\newcommand{\dka}{\texttt{Der}_K(\A)}
\newcommand{\lle}{\mathcal{L}}

\newcommand{\lrak}{\rm \texttt{LR}_{A K}}
\newcommand{\acoma}{\rm \texttt{AComm}_{\A}}
\newcommand{\liek}{\rm \texttt{Lie}_{K}}
\newcommand{\liea}{\rm \texttt{Lie}_{\A}}

\newcommand{\uea}{\universal_{\A}}
\newcommand{\lrf}{\rm \texttt{LR}}
\newcommand{\zena}{\Zenter_{\A}}
\newcommand{\uce}{\mathfrak{u}}
\newcommand{\ucea}{\mathfrak{uce}_\A}
\newcommand{\Vect}{_K{\bf mod}}

\newcommand{\derrin}{\mathfrak{Der}}

\DeclareMathOperator{\Rin}{Rin}
\DeclareMathOperator{\Lie}{Lie}
\DeclareMathOperator{\Der}{Der}
\DeclareMathOperator{\End}{End}
\DeclareMathOperator{\IDer}{IDer}
\DeclareMathOperator{\ab}{ab}
\DeclareMathOperator{\ad}{\sf ad}
\DeclareMathOperator{\MM}{\mathcal{M}}
\DeclareMathOperator{\modu}{mod}
\DeclareMathOperator{\Poiss}{Poiss}

\newtheorem{Th}{Theorem}[section]
\newtheorem{Pro}[Th]{Proposition}
\newtheorem{Le}[Th]{Lemma}
\newtheorem{Co}[Th]{Corollary}
\theoremstyle{definition}
\newtheorem{De}[Th]{Definition}
\newtheorem{Ex}[Th]{Example}
\theoremstyle{remark}
\newtheorem{Rem}[Th]{Remark}

\begin{document}

\title{Universal central extensions of Lie--Rinehart algebras}
\author{J. L. Castiglioni}
\address{[J. L. Castiglioni] Dpto. Matem\'atica, Facultad de Ciencias Exactas Univ. Nacional de La Plata and CONICET, 1900 La Plata, Argentina}
\email{jlc@mate.unlp.edu.ar}
\author{X. García--Martínez}
\address{[X. García--Martínez] Department of Algebra, University of Santiago de Compostela, 15782, Spain.}
\email{veigano@gmail.com}
\author{M. Ladra}
\address{[M. Ladra] Department of Algebra, University of Santiago de Compostela, 15782, Spain.}
\email{manuel.ladra@usc.es}

\thanks{The  first author was supported by Grant PIP 112-201101-00636-CONICET. The  second and third authors were supported by Grant of Xunta de Galicia GRC2013-045 (European FEDER support included).}

\begin{abstract}
In this paper we study the universal central extension of a Lie--Rinehart algebra and we give a description of it.  Then we study the lifting of automorphisms and derivations to central extensions. We also give a definition of a non-abelian tensor product in Lie--Rinehart algebras based on the construction of Ellis of non-abelian tensor product of Lie algebras. We relate this non-abelian tensor product to the universal central extension.

\end{abstract}

\maketitle

\section{Introduction}
Let A be a unital commutative algebra over a commutative ring $K$ with unit. A
Lie--Rinehart algebra is a Lie $K$-algebra, which is also an
$\A$-module and these two structures are related in an appropriate
way \cite{Hue}. The leading example of Lie--Rinehart algebras is
the set $\dka$ of all $K$-derivations of A. Lie--Rinehart algebras
are the algebraic counterpart of Lie algebroids \cite{Mac}.

The concept of Lie--Rinehart algebra generalizes the notion of Lie algebra. In \cite{Gar} universal central extensions of Lie algebra are studied, proving that if a Lie algebra is perfect then it has a universal central extension. Moreover, it is characterized the kernel of the universal central extension as de second homology group with trivial coefficients. In this paper we extend this study to Lie--Rinehart algebras.

On the other hand, in \cite{Ell1} a non-abelian tensor product of Lie algebras is introduced, its more important properties are studied and it is related to the universal central extension. In this paper we extend this construction to a non-abelian tensor product of Lie--Rinehart algebras, we study some important properties and we relate it to the universal central extension of Lie--Rinehart algebras.

After the introduction, the paper is organized in four sections. In Sec. \ref{S:Prel}, we
recall some needed notions and facts on Lie--Rinehart algebras, actions, crossed
modules, universal enveloping algebras, free algebras, homology and cohomology and abelian extensions of Lie--Rinehart algebras.
 In Sec. \ref{S:uce}, following Neher's paper on Lie superalgebras \cite{Neh},  we introduce central extensions and universal central extensions of Lie--Rinehart algebras giving a characterization of them (Theorem ~\ref{T:ucemain}), extending classic results of Lie algebras (see \cite{Gar}).
We construct an endofucntor $\ucea$ that when the Lie--Rinehart algebra is perfect gives explicitly the universal central extension.
In Sec. \ref{S:lift}, we study the lifting of automorphisms and derivations to central extensions.
Finally, in Sec. \ref{S:tensor}, we introduce a non-abelian tensor product of Lie--Rinehart algebras extending Ellis \cite{Ell1} non-abelian tensor product of Lie algebras. We relate this non-abelian tensor product with the universal central extension.

\section{Preliminaries on Lie--Rinehart algebras} \label{S:Prel}
Most of the content of this
section is well known, or follows from known results. We included it in
order to fix terminology, notations and main examples. In what
follows we fix a unital commutative ring $K$. All modules are
considered over $K$. We write $\t$ and $\Hom$ instead of $\t_K$
and $\Hom_K$.

\subsection{Definitions, Examples}
 Let A be a unital commutative algebra over
$K$. Then the set  $\dka$ of all
$K$-derivations of A is a Lie $K$-algebra and an $\A$-module
simultaneously. These two structures are related by the  following
identity
\[
[D,aD'] = a[D,D'] +D(a)D', \quad  D, D' \in \dka.
\]
This leads to the notion below,
which goes back to Herz under the name ``pseudo-alg\`ebre de Lie''
and which is the algebraic counterpart of the Lie algebroid
\cite{Mac}.
\begin{De} A \emph{Lie--Rinehart $\A$-algebra} consists of a Lie $K$-algebra
$L$ together with an $\A$-module structure on $L$ and a morphism, called the \emph{anchor} map,
\[\alpha \colon L \to \dka ,\]
 which is simultaneously a Lie algebra
and $\A$-module  homomorphism such that \[[x,ay] = a[x,y]+x(a)y.\]
Here $x,y \in L$, $a \in \A$ and we  write $x(a)$ for
$\alpha(x)(a)$ \cite{Hue}. These objects are also known as $(K,\A)$-Lie algebras \cite{Rin} and $d$-Lie rings \cite{Pal}.
\end{De}
 Thus  $\dka$ with $\alpha={\sf Id}_{\dka}$ is a
Lie--Rinehart $\A$-algebra. Let us observe that Lie--Rinehart
A-algebras with trivial homomorphism $\alpha \colon L \to \dka$ are
exactly Lie $\A$-algebras. Therefore the concept of Lie--Rinehart
algebras generalizes the concept of Lie A-algebras. If $\A = K$, then
$\dka$=0 and there is no difference between Lie and Lie--Rinehart
algebras. If $L$ is an A-module, then $L$ is a  trivial Lie--Rinehart $\A$-algebra, that is $L$ itself endowed with trivial Lie bracket and trivial anchor map.

If $L$ and $L'$ are Lie--Rinehart algebras, a \emph{Lie--Rinehart $\A$-algebra homomorphism} $f \colon L \to L'$ is a map, which is simultaneously a
Lie $K$-algebra homomorphism and a homomorphism of $\A$-modules. Furthermore it has to conserve the action on $\dka$, in other words the diagram
\[
\xymatrix{
L \ar[rr]^{f} \ar[dr]_-{\alpha} &  &    L' \ar[dl]^-{\alpha'}    \\
                          & \dka               }
\]
commutes.
We denote by $\lrak$ the category of
Lie--Rinehart A-algebras. We have the full inclusion
\[
\liea \subset \lrak,
\]
where $\liea$ denotes the category of Lie $\A$-algebras.

It is important to see that the product in this category is not the cartesian product. For two Lie--Rinehart algebras $L$ and $M$,
 the product in $\lrak$ is $L \times_{\dka} M = \{(l, m) \in L \times M : l(a) = m(a) \text{ for all } a \in \A \}$,
  where $L \times M$ denotes the cartesian product, with the action $(l, m)(a) = l(a) = m(a)$ for all $a \in \A$.

Note also that the category  $\lrak$  does not have zero object. This way, when we speak about a short exact sequence $I \to E \to L$ in $\lrak$, we mean that the first homomorphism is injective and the second is surjective.

Let $L$ be a Lie--Rinehart $\A$-algebra. A Lie--Rinehart subalgebra $M$ of $L$ is a $K$-Lie subalgebra which is an $\A$-module, with action induced by the inclusion in $L$.
If $M$ and $N$ are two Lie--Rinehart subalgebras of $L$, we define the \emph{commutator} of $M$ and $N$, denoted by $\{M, N\}$ as
the \emph{span} as a $K$-module of the elements of the form $a[x, y]$ where $a \in \A$, $x \in M$ and $y \in N$.
Given a subalgebra $M$ of $L$ we say that is an ideal if $M$ is $K$-Lie ideal of $L$ and the action induced by the inclusion is the trivial action.
 An example of ideal
 is the kernel of a Lie--Rinehart homomorphism. Another example is the center of a Lie--Rinehart algebra, defined by
\[
\zena(L) = \{ x\in L : [ax, z] = 0 \ \text{ for all } a \in \A, z \in L \}.
\]

Note that $\{L,L\}$ is a subalgebra of $L$ but it is not an ideal of $L$. We denote by $L^{\ab}$ the A-module $L/\{L,L\}$.

\begin{Ex} \label{trans} If ${\mathfrak g}$ is a $K$-Lie algebra acting
on a commutative $K$-algebra A by derivations (that is, a
homomorphism of Lie $K$-algebras
 $\gamma \colon {\mathfrak g} \to \dka$ is
given), then \emph{the
 transformation Lie--Rinehart algebra
of} $({\mathfrak g}, \A)$ is $L  = \A \otimes {\mathfrak
g}$ with the Lie bracket
\[
[a\otimes g,a'\otimes g'] : = aa'
\otimes [g,g']+a\gamma(g)(a')\otimes g'-a'\gamma(g')(a)\otimes g,
\]
where $a, a' \in \A$, $g, g' \in \mathfrak g$
and the action $\alpha \colon L \to \dka$ is given by $\alpha(a
\otimes g)(a') = a\gamma(g)(a')$.
\end{Ex}
\begin{Ex}\label{E:atiyah}
Let $\MM$ be an A-module. The \emph{Atiyah algebra} $\mathcal{A}_{\MM}$  of $M$ is the Lie--Rinehart A-algebra
whose elements are pairs $(f,D)$ with $f \in \End_K(\MM)$ and $D \in \dka$ satisfying the following property:
\[f(am)= af(m) +D(a)m, \quad a \in \A, m\in \MM. \]
$\mathcal{A}_{\MM}$ is a Lie--Rinehart A-algebra  with the Lie bracket
 \[ [(f,D),(f',D')]=([f,f'],[D,D'])\]
 and anchor map $(f,D) \mapsto  D$ (see \cite{Kal}).
\end{Ex}
\begin{Ex}
Consider the
$K$-algebra of dual numbers,
\[
\A=K[\varepsilon]=K[X]/(X^2)=\big\{ c_1+c_2 \varepsilon \mid c_1,c_2\in K,\varepsilon^2=0\big\}.
\]
 We can endow to A with the Lie algebra structure given by the bracket:
\[
[c_1+c_2\varepsilon,c'_1+c'_2 \varepsilon]=(c_1c'_2-c_2c'_1)\varepsilon, \quad c_1+c_2 \varepsilon,c'_1+c'_2 \varepsilon \in \A.
\]
Thus A is a  Lie--Rinehart A-algebra with anchor map $\alpha \colon \A \rightarrow  \dka$,
 $c_1+c_2 \varepsilon\mapsto  \ad_{c_1}$, where
 $\ad_{c_1}(c'_1+c'_2\varepsilon)=[c_1,c'_1+c'_2\varepsilon]$ is the adjoint map of $c_1$.
\end{Ex}

\begin{Ex}
The $\A$-module $\dka \oplus\A$ is a  Lie--Rinehart A-algebra with the  bracket
\[
[(D,a),(D',a') ]=\big([D,D'],D(a')-D'(a)\big),
\]
and anchor map  $\pi_1 \colon \dka  \oplus \A\rightarrow \dka$,  the projection onto the first factor.
\end{Ex}

\begin{Ex} \label{poislie}  Let us recall that a
\emph{Poisson algebra} is a commutative $K$-algebra $P$ equipped
with a Lie $K$-algebra structure such that the following identity
holds
\[
[a,bc]=b[a,c]+ [a,b]c,   \qquad a, b, c \in P.
\]
There are (at least) three Lie--Rinehart algebra related to $P$.
The first one is $P$ itself considered as a $P$-module in an
obvious way, where the action of $P$ (as a Lie algebra) on $P$ (as
a commutative algebra) is given by the homomorphism $\ad \colon P\to
\Der(P)$ given by $\ad(a)=[a,-]\in \Der(P)$. The second Lie--Rinehart algebra is the module of K\"ahler
differentials $\Omega ^1_{P}$. It is easily shown (see \cite{Hue})
that there is a unique Lie--Rinehart algebra structure on $\Omega
^1_{P}$ such that $[da,db]=d[a,b]$ and such that the Lie algebra
homomorphism $\Omega^1_P\to \Der(P)$ is given by
$adb\mapsto a[b,-]$. To describe the third one, we need some
preparations. We put
\[H^0_{\Poiss}(P,P):= \{a\in P \mid [a,-]=0\}.\]
Then $H^0_{\Poiss}(P,P)$ contains the unit of $P$ and is closed
with respect to products, thus it is a subalgebra of $P$. A
\emph{Poisson derivation} of $P$ is a linear map $D \colon P\to P$ which
is a simultaneous derivation with respect to commutative and Lie
algebra structures. We let ${\sf Der}_{\Poiss}(P)$ be the
collection of all Poisson derivations of $P$. It is closed with
respect to Lie bracket. Moreover if $a\in H^0_{\Poiss}(P,P)$ and
$D\in {\sf Der}_{\Poiss}(P)$ then $aD \in {\sf Der}_{\Poiss}(P)$. It
follows that ${\sf Der}_{\Poiss}(P)$ is a Lie--Rinehart
$H^0_{\Poiss}(P,P)$- algebra.
 There is the following variant of the first
construction in the graded case. Let $P_*=\bigoplus _{n\geq 0}P_n$
be a commutative graded $K$-algebra in the sense of commutative
algebra (i.e., no signs are involved) and assume $P_*$ is equipped
with a Poisson algebra structure such that the bracket has degree
$(-1)$. Thus $[-,-] \colon P_n\t P_m\to P_{n+m-1}$. Then $P_1$ is a
Lie--Rinehart $P_0$-algebra, where the Lie algebra homomorphism
$P_1\to \Der(P_0)$ is given by $a_1 \mapsto [a_1,-]$,
$[a_1,-]( a_0)=[a_1,a_0]$, where $a_i\in P_i$, $i=0,1$.
\end{Ex}

\subsection{Actions and Semidirect Product of Lie--Rinehart algebras}

\begin{De}
Let $L\in \lrak$ and let $R$ be a Lie $\A$-algebra. We will say that $L$ \emph{acts on} $R$ if it is given a $K$-linear
map
\[
L \otimes R \rightarrow R, \ (x,r)\mapsto  x \circ r, \ x\in L,r \in R
\]
such that the following identities hold
\begin{enumerate}
\item[1)] $[x,y] \circ r = x \circ (y \circ r) -y \circ (x \circ r)$,

\item[2)] $x \circ [r,r'] = [x \circ r,r'] - [x \circ r',r]$,

\item[3)] $ax \circ r = a(x \circ r)$,

\item[4)] $x \circ (ar) = a(x \circ r)+x(a)r$,
\end{enumerate}
where $a \in \A$, $x, y \in L$ and $r, r' \in R$.
\end{De}
Let us observe that 1) and 2) mean that $L$ acts on $R$
in the category of Lie $K$-algebras.

Let us consider a Lie--Rinehart algebra $L$ and a Lie $\A$-algebra
$R$ on which $L$ acts. Since $L$ acts on $R$
in the category of
Lie $K$-algebras as well, we can form the \emph{semi--direct product} $L \rtimes R$
in the category of Lie $K$-algebras, which is
$L\oplus R$ as a $K$-module, equipped with the
following bracket
\[
\lbrack (x,r),(y,r')\rbrack :=([x,y],[r,r']+ x \circ r'-y \circ r),
\]
where $x, y \in L$ and $r, r' \in R$.
We claim that $L\rtimes R$ has also a natural
Lie--Rinehart algebra structure. Firstly, $L\rtimes R$
as an $\A$-module is the direct sum of $\A$-modules
$L$ and $R$. Hence $a(x,r)=(ax,ar)$. Secondly the map
\[\widetilde{\alpha} \colon L \rtimes R\rightarrow {\dka}\]
is given by $\widetilde{\alpha }(x,r):=\alpha (x)$. In this way we
really get a Lie--Rinehart algebra. Indeed, it is clear that
$\widetilde{\alpha }$ is simultaneously an $\A$-modules and Lie
algebras homomorphism and it is obtained
\begin{align*}
[(x,r),a(y,r')] &=
[(x,r),(ay,ar')]
=([x,ay], [r,ar']+
x \circ (ar')- ay \circ r)\\
{}&=\big(a[x,y] +
x(a)y,a[r,r']+a(x \circ r')+x(a)r'-a(y \circ r)\big)\\
{}&=a([x,y],[r,r']+x \circ r'-y \circ r)+
(x(a)y,x(a)r')\\
{}&=a[(x,r),(y,r')]+x(a)(y,r').
\end{align*}
Thus $L \rtimes R$ is indeed a Lie--Rinehart algebra.

\begin{De}[\cite{Pal}] \emph{A left Lie--Rinehart $(\A,L)$-module} over a Lie--Rinehart A-algebra $L$ is a
$K$-module $\MM$ together with two operations
\[{L\t \MM \to \MM}, \qquad (x,m)\mapsto xm,\] and
\[\A\t \MM\to \MM, \qquad (a,m)\mapsto am ,\] such that the
first one makes $\MM$ into a module over the Lie $K$-algebra $L$ in the sense
of the Lie algebra theory, while the second map makes $\MM$ into an A-module and
additionally the following compatibility conditions hold
\begin{align*}
(ax)(m)&=a(xm), \\
x(am)&=a(xm)+x(a)m,  \qquad a\in \A, m\in \MM \ \text{and} \ x\in L.
\end{align*}
\end{De}

That is, $M$ is an abelian Lie A-algebra and $L$ acts on $M$.

Notice that a left Lie--Rinehart $(\A,L)$-module is equivalent to give a morphism of Lie--Rinehart A-algebras  $L \to \mathcal{A}_{\MM}$ (see Example~\ref{E:atiyah}).

It follows that $\A$ is a left Lie--Rinehart $(\A,L)$-module for any Lie--Rinehart
algebra $L$ given by the anchor.
\begin{De}[\cite{Hue1}] \emph{A right Lie--Rinehart $(\A,L)$-module} over a Lie--Rinehart A-algebra $L$ is a
$K$-module $\MM$ together with two operations
\[{\MM \t L \to \MM}, \qquad (m,x)\mapsto mx,\] and
\[\A\t \MM\to \MM, \qquad (a,m)\mapsto am ,\] such that the
first one makes $\MM$ into a module over the Lie $K$-algebra $L$ in the sense
of the Lie algebra theory, while the second map makes $\MM$ into an A-module and
additionally the following compatibility conditions hold
\[(am) x= m(ax)=a(mx)-x(a)m,  \qquad a\in \A, m\in \MM \ \text{and} \ x\in L.\]
\end{De}
\begin{Rem}\label{rem}
The differences between the definitions of left and right $(\A, L)$-module are significantly large. While in Lie algebras left and right $L$-modules are equivalent, in Lie--Rinehart that is not true. Concretely,  $\A$ has a canonical left $(\A, L)$-module structure but it does not hold a canonical right $(\A, L)$-module structure. See \cite{Hue2} for a characterization of right $(\A, L)$-module structures and see \cite{KrRo} for a concrete example.
\end{Rem}

\subsection{Crossed Modules of Lie--Rinehart algebras}

A \emph{crossed module $\partial \colon R \to L$ of Lie--Rinehart $\A$-algebras} (see \cite{CLP1}) consists of a Lie--Rinehart algebra $L$ and a Lie
 $\A$-algebra $R$ together with the action of $L$ on $R$ and the Lie $K$-algebra homomorphism $\partial$ such that the following identities hold:
\begin{enumerate}
\item $\partial(x \circ r) = [x, \partial(r)]$,
\item $\partial(r')\circ r = [r', r]$,
\item $\partial(ar) = a\partial(r)$,
\item $\partial(r)(a) = 0$,
\end{enumerate}
for all $a \in \A, r \in R$ and $x \in L$.

We can see some examples of crossed modules of Lie--Rinehart algebras.
\begin{enumerate}
\item For any Lie--Rinehart homomorphism $f \colon L \to R$, the diagram $\Ker f \to L$ is a crossed module of Lie--Rinehart algebras.

\item If $M$ is an ideal of $L$, the inclusion $M \hookrightarrow L$ is a crossed module where the action of $L$ on $M$ is given by the Lie bracket.

\item If $R$ is a left Lie--Rinehart $(\A,L)$-module, the morphism $0 \colon R \to L$ is a crossed module.

\item If $\partial \colon R \to L$ is a central epimorphism \big(i.e. $\Ker \partial \subset \Zenter(R)$\big) from a Lie $\A$-algebra $R$
 to a Lie--Rinehart algebra $L$, $\partial$ is a crossed module where the action from $L$ to $R$ is given by $x \circ r = [r', r]$, such that $\partial(r') = x$.
\end{enumerate}

\subsection{Universal enveloping algebras and related constructions}
There is a $K$-algebra $\uea L$ that has the property that the category of left
$\uea L$-modules is equivalent to the category of
left $(\A,L)$-modules. Actually this algebra was constructed in \cite{Rin}. We define the algebra
$\uea L$ in terms of generators and relations. We have
generators $i(x)$ for each $x \in L$ and $j(a)$ for each $a\in \A$. These generators must satisfy the following relations
\begin{align*}
j(1)&=1, \qquad \qquad  j(ab)=j(a)j(b),\\
i(ax)& =j(a)i(x),\\
i([x,y])&= i(x)i(y)-i(y)i(x),\\
i(x)j(a)&=j(a)i(x)+j\big(x(a)\big).
\end{align*}
The first relations show that $j \colon \A \to \uea L$ is an
algebra homomorphism.

Notice that in case of a trivial anchor one obtains the universal enveloping algebra of $L$ as a Lie A-algebra.

We let $V_n$ be the $\A$-submodule spanned on
all products $i(x_1)\cdots i(x_k)$, where $k\leq n$. Then
\[0\subset \A = V_0 \subset V_1 \subset \cdots \subset V_n \subset \cdots \subset \uea L  \]
defines an algebra filtration on $\uea L$. It is
clear that   $\uea L = \cup_{n\geq 0}V_n$. It follows
from the third relation  that the associated graded object ${\sf
gr}_*(V)$ is a commutative $\A$-algebra. In other words,
$\uea L$ is an almost commutative $\A$-algebra in the
following sense.

An \emph{almost commutative $\A$-algebra} is an associative $K$-algebra
$C$ together with a filtration
\[0\subset \A=C_0\subset C_1\subset \cdots \subset C_n\subset \cdots  \subset  C
= \bigcup_{n\geq 0}C_n\] such that $C_nC_m\subset C_{n+m}$ and
such that the associated graded object ${\sf gr}_*(C)=\bigoplus
_{n\geq 0}C_n/C_{n-1}$ is a commutative $\A$-algebra.
\begin{Rem}
It is
well known that if $C$ is an almost commutative $\A$-algebra, then
there is a well-defined bracket
\[[-,-] \colon {\sf gr}_n(C)\t {\sf gr}_m(C) \longrightarrow {\sf gr}_{n+m-1}(C)\]
which is given as follows. Let $a\in {\sf gr}_n(C)$ and $b\in{\sf
gr}_m(C) $ and $\hat{a}\in C_n$ and $\hat{b}\in C_m$ representing
$a$ and $b$ respectively. Since ${\sf gr}_*(C)$ is a commutative
algebra it follows that $\hat{a}\hat{b}-\hat{b}\hat{a}\in
C_{n+m-1}$ and the corresponding class in ${\sf gr}_{n+m-1}(C)$ is
$[a,b]$. It is also well known that in this way we obtain a
Poisson algebra structure on ${\sf gr}_*(C)$. Since the bracket is
of degree (-1) it follows from Example \ref{poislie} that
$L={\sf gr}_ 1(C)$ is a Lie--Rinehart $\A={\sf
gr}_0(C)$-algebra. Moreover the short exact sequence
\[\A \to C_1\to L  \]
is an abelian extension of Lie--Rinehart algebras (see below Definition \ref{D:ab_ext}).
\end{Rem}

\begin{Pro}\label{P:LR_functor}The correspondence assigning $C_1$
to the almost commutative $\A$-algebra $C$, defines a functor $LR \colon
{\normalfont \acoma} \to {\normalfont \lrak}$.
\end{Pro}

\begin{proof}
Let $f \colon C \to D \in \acoma$. Since $f$ preserves the
filtration, $f(C_1) \subseteq D_1$. Furthermore, $f(a x) =
f(a)f(x) = a f(x)$, for any $a \in C_0 = D_0$ and $x \in C_1$, and
$f([x,y]) = f(xy - yx)=f(x)f(y) - f(y)f(x) = [f(x),f(y)]$, for
$x,y \in C_1$. Hence the restriction of $f$ to $C_1$, which we
shall call $LR(f)$, is a morphism of $K$-Lie algebras and of
$\A$-modules such that the following diagram commutes in $\liek$,
\begin{equation*}
\xymatrix{
  C_1 \ar[rr]^{LR(f)} \ar[dr]_{[\circ, -]}
                &  & D_1  \ar[dl]^{[\circ, -]}    \\
                & \dka                }
\end{equation*}
Thus, $LR(f) \in \lrak$.

On the other hand, it is clear that $LR(1_{C}) = 1_{C_1}$ and
the following diagram commutes in $K$-mod,
\begin{equation*}
 \xymatrix{
   C  \ar[r]^{f} & D \ar[r]^{g} & E  \\
   C_1 \ar@{^{(}->}+<0ex,2ex>;[u]^-{i_C} \ar[r]^{LR(f)} & D_1
   \ar@{^{(}->}+<0ex,2ex>;[u]_-{i_D} \ar[r]^{LR(g)} & E_1 \ar@{^{(}->}+<0ex,2ex>;[u]_-{i_E}  }
\end{equation*}
Hence $LR$ is functorial.
\end{proof}

\begin{Pro}
The functor $LR$ is right adjoint to the universal enveloping
functor ${\normalfont \uea} \colon {\normalfont \lrak} \to {\normalfont \acoma}$.
\end{Pro}

\begin{proof}
Let $\Phi \colon \acoma(\uea L, C) \to \lrak\big(L, LR(C)\big)$
be the map given as follows. Since $\uea L$ is generated as a
$K$-algebra by $L$ and $\A$, a morphism $f \colon \uea L \to C$  is
completely determined by its restriction to $L$ and A. Since
$f(a) = a$ for every $a \in \A$, and $f(L) \subseteq f\big((\uea L)_1\big)
\subseteq C_1$, it follows that the restriction of $f$ to $L$,
$\Phi f \colon L \to C_1 = LR(C)$ is a monomorphism of Lie--Rinehart
algebras.

Let $g \colon L \to C_1 \in \lrak$. We build up $\widetilde{g} \colon \uea L \to C$ by
$\widetilde{g}(a x_1 \cdots x_m):= a g(x_1) \cdots g(x_m) \in C$. It is
straightforward to see that $\widetilde{g} \in \acoma$ and $\Phi \widetilde{g}= g$.
Hence $\Phi$ is bijective, and $\uea$ and $\lrf$ form an adjoint
pair.
\end{proof}

\subsection{Free Lie--Rinehart Algebras}
We have the functor
\[
U \colon \lrak \to {\mathrm { \Vect / \dka}}
\]
which assigns $\alpha \colon L\to \dka$ to a
Lie--Rinehart algebra $L$. Here ${\mathrm { \Vect / \dka}}$  is the category
of $K$-linear maps $\psi \colon V \to \dka$, where $V$ is a $K$-module.
 A morphism $\psi\to \psi_1 $ in ${\mathrm { \Vect / \dka}}$  is a
$K$-linear map $f \colon V\to V_1$ such that $\psi=\psi_1\circ f$. Now we
construct the functor
\[
F \colon {\mathrm {\Vect/\dka}}\to\lrak
\]
as follows. Let $\psi \colon V\to \dka$ be a $K$-linear map. We let ${\mathbf
L}(V)$ be the free Lie $K$-algebra generated by $V$. Then we have
the unique Lie $K$-algebra homomorphism ${\mathbf L}(V)\to  \dka$
which extends the map $\psi$, which is still denoted by $\psi$.
Now we can apply the construction from Example \ref{trans} to get
a Lie--Rinehart algebra structure on $ \A \otimes {\mathbf L}(V)$. We let
$F(\psi)$ be this particular Lie--Rinehart algebra and we call it
\emph{the free Lie--Rinehart algebra generated by $\psi$}. In this
way we obtain the functor $F$, which is the left adjoint to $U$.

Kapranov \cite{Kap} defines a different concept of free Lie--Rinehart algebra as the adjoint of the forgetful functor $U' \colon \lrak \to {\mathrm { _\A{\bf mod} / \dka}}$. The relation between both constructions is given in \cite[(2.2.8) Proposition]{Kap}.

\subsection{Rinehart homology and cohomology of Lie--Rinehart algebras}
Let $\MM$ be a left Lie--Rinehart $(\A,L)$-module. Let us recall the
definition of the Rinehart  cohomology  $H^{*}_{\Rin}(L,\MM)$ of a
Lie--Rinehart algebra $L$ with coefficients in a Lie--Rinehart
module $\MM$ (see \cite{Pal,Rin} and \cite{CLP2,Hue}). We put
\[
 C^n_\A(L,\MM) := \Hom_{\A} (\Lambda^n_{\A}L,\MM), \qquad \, n\geq 0,
\]
where $\Lambda^*_\A(V)$ denotes the exterior algebra over A
generated by an $\A$-module $V$. The coboundary map
\[
\delta  \colon C^{n-1}_\A(L,\MM)
\longrightarrow C^n_\A(L,\MM),
\]
is given by
\begin{multline*}
(\delta f)(x_1, \dots,x_n) = \sum_{i=1}^n(-1)^{(i-1)} \, x_i\big(f(x_1,
\dots , \hat{x_i}, \dots,x_n)\big) \\
+ \sum_{j<k} (-1)^{j+k} \, f([x_j,x_k],x_1, \dots,
\hat{x_j}, \dots,\hat{x_k}, \dots,x_n),
\end{multline*}
 where $x_1, \dots,x_n \in
L, m \in \MM, f \in C^{n-1}_\A(L,\MM)$.

We note that the differential  $\delta$ is not A-linear unless $L$ acts trivially on A.

For any left Lie--Rinehart $(\A,L)$-module  $\MM$, the \emph{Lie--Rinehart cohomology} is defined by
\[
H^{n}_{\Rin}(L,\MM) = H^n\big(C^n_\A(L,\MM) \big), \qquad n\geq 0.
\]

Let $\MM$ be a right Lie--Rinehart $(\A,L)$-module. Let us recall the
definition of the Rinehart homology  $H_{*}^{\Rin}(L,\MM)$ of a
Lie--Rinehart algebra $L$ with coefficients in a Lie--Rinehart
module $\MM$. We put
\[
C_n^\A(L,\MM)  := \MM \otimes_\A \Lambda^n_{\A}L,  \qquad \qquad n\geq 0.\\
\]
 The boundary  map
\[
\partial  \colon C_n^\A(L,\MM)
\longrightarrow C_{n-1}^\A(L,\MM), \\
\]
 is given by
\begin{multline*}
\partial\big( m \otimes_\A (x_1, \dots,x_n) \big)= \sum_{i=1}^n(-1)^{(i-1)} \, mx_i \otimes_\A (x_1,
\dots , \hat{x_i}, \dots,x_n) \\
+  \sum_{j<k} (-1)^{j+k} \, m \otimes_\A ([x_j,x_k],x_1, \dots,
\hat{x_j}, \dots,\hat{x_k}, \dots,x_n),
\end{multline*}
 where $x_1, \dots,x_n \in
L, m \in \MM$.

We note that the differential $\partial$  is not A-linear unless $L$ acts trivially on A.

For any right Lie--Rinehart $(\A,L)$-module  $\MM$,  the \emph{Lie--Rinehart homology} is defined by
\[
H_{n}^{\Rin}(L,\MM) = H_n\big (C_n^\A(L,\MM)\big), \qquad n\geq 0 .
\]

Let $\mathfrak g$ be a Lie algebra over $K$ and let $\MM$ be a $\mathfrak g$-module. Then
we have the Chevalley--Eilenberg chain and cochain complexes  $C_*^{\Lie}({\mathfrak g},\MM)$ and $C^*_{\Lie}({\mathfrak g},\MM)$,
which compute the Lie algebra (co)homology (see \cite{CaEi}):
\begin{align*}
C_n^{\Lie}({\mathfrak g},\MM)&=\Lambda ^n({\mathfrak g})\otimes \MM,\\
C^n_{\Lie}({\mathfrak g},\MM)&=\Hom(\Lambda ^n({\mathfrak g}),\MM).
\end{align*}
Here $\Lambda ^*$  denotes the exterior algebra defined over $K$.

 One observes that if $\A=K$, then $H_{*}^{\Rin}(L,\MM)$ and  $H^{*}_{\Rin}(L,\MM)$ generalize the classical
 definition of Lie algebra (co)homology.

For a general A by forgetting the $\A$-module structure one obtains the canonical
homomorphisms
\[H_{*}^{\Lie}(L,\MM) \to H_{*}^{\Rin}(L,\MM), \qquad  H^{*}_{\Rin}(L,\MM)\to H^{*}_{\Lie}(L,\MM),\]
where $H_{*}^{\Lie}(L,\MM)$ and $H^{*}_{\Lie}(L,\MM)$ denote  the homology and  cohomology of $L$
considered as a Lie $K$-algebra. On the other hand if
 A is a smooth commutative algebra, then
$H^{*}_{\Rin}(\Der(\A), \A)$ is isomorphic to the de Rham
cohomology of A (see \cite{Rin} and \cite{Hue}).

\begin{Le}\label{L:trans} Let ${\mathfrak g}$  be a Lie $K$-algebra acting on a
commutative algebra $\A$ by derivations and let $L$ be the
transformation Lie--Rinehart algebra of $({\mathfrak g},\A)$ (see Example \ref{trans}).
 Then for any   Lie--Rinehart $(\A,L)$-module
$\MM$ we have the canonical isomorphisms of
complexes $C_*^\A(L,\MM)\cong C_*^{\Lie}({\mathfrak g},\MM)$, $C^*_\A(L,\MM)\cong C^n_{\Lie}({\mathfrak g},\MM)$ and
in particular the isomorphisms
\begin{align*}
H_*^{\Rin}(L, \MM) & \cong  H_*^{\Lie}({\mathfrak g},\MM),\\
H^*_{\Rin}(L, \MM) & \cong  H^*_{\Lie}({\mathfrak g},\MM).
\end{align*}
\end{Le}

\begin{proof}
Since $L = \A \otimes {\mathfrak g} $ we have
$\Lambda ^n_\A L \otimes_\A \MM \cong \Lambda ^n_\A{\mathfrak g} \otimes_\A \MM $ and
$\Hom_\A(\Lambda ^n_\A L,\MM)\cong \Hom(\Lambda
^n{\mathfrak g},\MM) $ and lemma follows.
\end{proof}

\begin{Pro} Let $L$ be a free Lie--Rinehart
algebra generated by $\psi \colon V\to {\texttt{\rm Der}_K(\A)}$ and let $\MM$ be any
Lie--Rinehart  $(\A,L)$-module. Then
\begin{align*}
H_n^{\Rin}(L,\MM)& =0, \qquad n>1, \\
H^n_{\Rin}(L, \MM)&=0,  \qquad n>1.
\end{align*}
\end{Pro}

\begin{proof}
By our construction  $L$ is a transformation
Lie--Rinehart algebra of $({\mathbf L}(V),\A)$. Thus we can
apply Lemma \ref{L:trans} to get isomorphisms
$H_*^{\Rin}(L,\MM)\cong H_*^{\Lie}({\mathbf L}(V),\MM)$ and $H^*_{\Rin}(L,\MM)\cong H^*_{\Lie}({\mathbf L}(V),\MM)$ and then we can
use the well-known  vanishing result for free Lie algebras.
\end{proof}

\subsection{Low degree homology groups of Lie--Rinehart algebras}

By definition, $H_{0}^{\Rin}(L,\MM)=\frac{\MM}{\MM \circ L}$, is the module of \emph{coinvariants} of $\MM$, where $\MM \circ L$ means
 the $K$-submodule of $\MM$ generated by $mx$, $x \in L, m\in \MM$, and $H^{0}_{\Rin}(L,\MM)=\MM^L=\{m \in \MM \mid x m=0 \ \text{for all}
  \ x\in L\}$, is the \emph{invariant $K$-submodule} of $\MM$.

 It follows from the definition that one has the following exact sequence
\begin{equation}\label{h01}
0\to H^0_{\Rin}(L,\MM)\to \MM \xrightarrow{ \ d \ }   \Der_\A(L,\MM)\to
H^1_{\Rin}(L,\MM)\to 0,
\end{equation}
where $\Der_\A(L,\MM)$ consists of A-linear maps $d \colon L\to
\MM$ which are derivations from the Lie $K$-algebra $L$ to $\MM$. In other
words $d$ must satisfy the following conditions:
\begin{align*}
d(ax) & = ad(x), \\
d([x,y])&=x\big(d(y)\big)-y\big(d(x)\big), \qquad  \quad  a\in \A,\, x,y\in L.
\end{align*}
If $m \in \MM$, the map $d_m \colon L \to \MM, x \mapsto xm$, is a derivation.
The maps $d_m$ are called \emph{inner derivations }of $L$ into $\MM$, and they form an $K$-submodule $\IDer_\A(L,\MM)$ of $\Der_\A(L,\MM)$. By \eqref{h01},
$H^1_{\Rin}(L,\MM)\cong \Der_\A(L,\MM)/\IDer_\A(L,\MM)$.

If $\MM$ is a trivial $(\A,L)$-module, then $H^1_{\Rin}(L,\MM)\cong \Der_\A(L,\MM) \cong \Hom_\A(L^{\ab},\MM)$ and
$H_1^{\Rin}(L,\MM)\cong \frac{\MM \otimes_\A L}{\MM \otimes_\A \{L,L\}} \cong \MM \otimes_\A L^{\ab}$.

\subsection{Abelian  extensions of Lie--Rinehart algebras}

 \begin{De}\label{D:ab_ext}
Let $L$ be a Lie--Rinehart A-algebra
and let $\MM$  a  left Lie-Rinehart $(\A,L)$-module. An \emph{abelian extension} of $L$
by $\MM$ is  a short exact sequence
\[  \MM \xrightarrow { \ i \ } L' \xrightarrow { \ \partial \ } L, \]
where $L'$ is a Lie--Rinehart $\A$-algebra and $\partial$ is a
Lie--Rinehart algebra homomorphism. Moreover, $i$ is an $\A$-linear map and the
following identities hold
\begin{align*}
[i(m),i(n)]&=0,\\
[i(m),x']&=\big(\partial(x')\big)(m), \qquad  m,n\in \MM, \ x'\in L'.
\end{align*}
An abelian extension is
called $\A$-\emph{split} if $\partial$ has an $\A$-linear section.
\end{De}

\begin{Pro}[{\cite[Theorem 2.6]{Hue}}]\label{P:ab_ext}
If $L$ is $\A$-projective, then the cohomology $H^2_{\Rin}(L,\MM)$ classifies the abelian
extensions
\[\MM  \xrightarrow { \  \ } L ' \xrightarrow { \  \ }  L \]
of $L$ by $\MM$ in the category of Lie--Rinehart  algebras which
split in the category of $\A$-modules.
\end{Pro}

For a left Lie-Rinehart $(\A,L)$-module $\MM$
one can define \emph{the semi-direct product} $L\rtimes \MM$ to
be $L \oplus \MM$ as an A-module with the bracket
$[(x,m),(y,n)]=\big([x,y],x n-y m\big)$, $x,y\in L, m,n  \in \MM$.

The extension $ \MM  \xrightarrow { \  \ } L\rtimes \MM \xrightarrow { \  \ }  L $
represents $0 \in H^2_{\Rin}(L,\MM)$.

\section{Universal central extensions of Lie--Rinehart algebras}
\label{S:uce}

\subsection{Central extensions}

An \emph{extension} of a Lie--Rinehart algebra $L$ is a short exact sequence
\begin{equation}
\label{secext}
 I \xrightarrow { \  i \ }   E \xrightarrow { \ p  \ }   L,
\end{equation}
where $I, E$ and $L$ are Lie--Rinehart algebras and $i, p$ are Lie--Rinehart homomorphisms.
Since $i \colon  I \to i(I) = \Ker p$ is an isomorphism we shall identify
$I$ and $i(I)$. In other words, an extension of $L$ is an
surjective Lie--Rinehart homomorphism $p \colon  E \to L$. If $p \colon  E \to L$ and $p' \colon  E' \to L$ are
two extensions of $L$, a \emph{homomorphism} from $p$ to $p'$ is a
commutative diagram in $\lrak$ of the form
\[
\xymatrix{
  E \ar[rr]^{f} \ar@{>>}[dr]_-{p}
                &  &    E' \ar@{>>}[dl]^-{p'}    \\
                & L                 }
\]

In particular,
\[\Ker f \subseteq f^{-1}(\Ker p') = \Ker p \ \textrm{and} \ E' = f(E) + \Ker p' .\]

An extension \eqref{secext} is called \emph{split} if there exists
a Lie--Rinehart morphism $s \colon  L \to E$, called \emph{splitting
homomorphism}, such that $p s = 1_L$. In this case, $E = I \oplus
s(L)$ and $s \colon  L \to s(L)$ is an isomorphism with inverse $f|s(L)$.
Moreover, $E \simeq I \rtimes L$, the semidirect product. In this
way, semidirect products and split exact sequences are in a one to
one correspondence. We point out that not every extension splits.
We shall say that an extension \emph{splits uniquely}  whenever
the splitting morphism is unique.

A \emph{central extension of $L$} is an extension $p$ such that
$\Ker p \subseteq \zena (E)$. In particular, if $p \colon
\xymatrix@C=0.6cm{E \ar@{>>}[rr]_{p} & & L \ar@/_/[ll]_{s}}$ is a split
central extension, it is a direct product of $K$-Lie algebras $E = \Ker p \times L$,  which is also a Lie--Rinehart algebra.
\begin{Pro}
If $L$ is $\A$-projective, then  $H^2_{\Rin}(L,I)$ classifies the central
extensions
\[ I  \xrightarrow { \  \ } E \xrightarrow { \  \ }  L \]
of $L$ by $I$.
\end{Pro}
\begin{proof}
Note that, if $I$ is a trivial left Lie--Rinehart $(\A,L)$-module, then an abelian extension of $L$ by $I$ is a central extension, and so the assertion follows by Proposition \ref{P:ab_ext}.
\end{proof}

A  Lie--Rinehart $\A$-algebra  $L$ is said \emph{perfect} if $L=\{L,L\}$.
A central extension $E$ of $L$ is called a \emph{covering} if $E$ is
perfect; in that case, $L$ is also perfect.

A central extension $\uce \colon  \lle \to L$ is called \emph{universal} if there exists a unique
homomorphism from $\uce$ to any other central extension of $L$.
From the universal property of universal central extensions it
immediately follows that two universal central extensions of $L$
are isomorphic as extensions.

\begin{Le} \textbf{(central trick)}
\label{ctrick}
Let $p\colon  E \xtwoheadrightarrow { \ }L$  be a central extension.
\begin{itemize}
\item[(a)] If $p(x) = p(x')$ and $p(y) = p(y')$ then $[x,y] = [x', y']$
and for every $a \in \A$, $x(a) = x'(a)$.
\item[(b)] If the following
diagram commutes in ${\normalfont \lrak}$,
\[\xymatrix@C=0.7cm{
P \ar@<-0.5ex>[r]_{f} \ar@<0.5ex>[r]^{g} & E \ar@{>>}[r]^{p} & L }
\]
then the restriction of both $f$ and $g$ to $\{P,P\}$ agree; i.e.,
$f|_{\{P,P\}} = g|_{\{P,P\}}$.
\end{itemize}
\end{Le}

\begin{proof}
(a) We have $x' = x + z$ and $y' = y + z'$ for some $z, z' \in \Ker p \subset \zena E$, so it is clear that $[x', y'] = [x + z, y + z'] = [x, y]$.
 In addition, if $p$ is a Lie--Rinehart homomorphism, the action on $\dka$ must be preserved so $x(a) = x(a')$.

(b) Using part (a), we have $g(a[x,y]) = a[g(x), g(y)] = a[f(x), f(y)] = f(a[x,y])$.
\end{proof}

\begin{Le}
\label{perfect}
Let $p\colon  E \xtwoheadrightarrow { \ }L$ be a central extension where $L$ is perfect. Then
\begin{enumerate}
\item[(a)] $E = \{E,E\} + \Ker p$, and $p' = p_{\{E,E\}} \colon \conm{E} \xtwoheadrightarrow { \ } L$ is a covering.
\item[(b)] $\zena E = p^{-1}(\zena L)$ and $p(\zena E) = \zena L$.
\item[(c)] If $f \colon L \xtwoheadrightarrow { \ } M $ is a central extension then so is $fp \colon E \xtwoheadrightarrow { \ } M$.
\item[(d)] If $f \colon C \xtwoheadrightarrow { \ } L$ is a covering and
\[
\xymatrix{
  E \ar[rr]^{g} \ar@{>>}[dr]_-{p}
                &  &    C \ar@{>>}[dl]^-{f}    \\
                & L                 }
\]
a morphism of extensions, then $g \colon E \xtwoheadrightarrow { \ } C$ is a central extension. In particular,
$g$ is surjective.
\end{enumerate}
\end{Le}

\begin{proof}
(a) Since $p(\conm{E}) = \conm{L} = L$ it follows easily that $E = \conm{E} + \Ker p$ and $p_{\{E,E\}}$ is clearly a covering.

(b) Let $z \in \zena(A)$. For every $a \in \A$ we have $[az, E] = 0$, so $0 = [ap(z), p(E)] = [ap(z), L]$ then $p(z) \in \zena(L)$.
 Conversely, let $z \in p^{-1}\big(\zena(L)\big)$. For every $a \in \A$ we have $p([az, E]) = [ap(z), L] = 0$ so
  $[az, E]\subset \Ker p\subset \zena(E)$. Since $[az, E] = [az, \conm{E} + \Ker p] = [az, \conm{E}]$ we just have to
   check that $[az, \conm{E}]$ is zero. Therefore, $\big[az, b[x,y]\big] = b\big[az, [x,y]\big] = b\big[x, [az, y]\big] + b\big[y, [x, az]\big] = 0$.

(c) It is clearly surjective and $\Ker fp = p^{-1}(\Ker f) \subset p^{-1}\big(\zena(L)\big) = \zena(E)$.

(d) By Lemma \ref{ctrick}(b) we have that $C = \conm{C} = \conm{g(E)} = g(\conm{E})$ so $g$ is surjective. Moreover, it is central since $\Ker g \subset \Ker p$.
\end{proof}

\begin{Co}
\label{corzentr}
Let $L \in {\normalfont \lrak}$, arbitrary. If $L/\zena L$ is perfect, then $\zena(L/\zena L) = 0$.
\end{Co}

\begin{proof}
It can be seen applying the second formula of Lemma \ref{perfect}(b) to the canonical map $p \colon \xymatrix@C=0.5cm{L \ar@{>>}[r] & L/\zena L}$, which is a central extension.
\end{proof}

\begin{Le} \textbf{(pullback Lemma)}
\label{pblemma}
Let $c \colon  N \xtwoheadrightarrow { \ } M$ be a central extension and $f \colon L \to M$ a morphism of Lie--Rinehart
algebras, then,
\[P := \{(l,n) \in L \times_{\normalfont \dka} N : f(l) = c(l) \} \]
is a Lie--Rinehart algebra and $p_L \colon P \to L$, $(l,n) \mapsto l$, is a central extension.
This extension splits if and only if there exists a (unique) Lie--Rinehart morphism $h \colon L \to N$
such that $ch = f$.

\begin{displaymath}
\xymatrix{ P \ar[r]^{p_N} \ar@{>>}[d]^{p_L}                                       & N \ar@{>>}[d]^c \\
              L \ar [r]_f  \ar@{.>}[ru]_h   \ar@/^/@{.>}[u]^s                   & M  }
\end{displaymath}
\end{Le}

\begin{proof}
It is clear that $P$ is a Lie--Rinehart algebra with action $(l,n)(a) = l(a) = n(a)$, and $p_L$ is a central extension.
 Moreover, a splitting homomorphism $s \colon L \to P$ exists (uniquely) if and only if there exists a (unique) Lie--Rinehart
 homomorphism $h \colon L \to N$ such that $s(l) = \big(l, h(l)\big)$ for all $l \in L$.
\end{proof}

\begin{Th} \textbf{(characterization of universal central extensions)}
\label{T:ucemain}
For a Lie--Rinehart algebra $L$, there are equivalent:
\begin{enumerate}
\item \label{1} Every central extension $L' \to L$ splits uniquely.
\item \label{2} $1_L \colon L \to L$ is a central extension. \\[.2cm]
\noindent If $\uce \colon L \to M$ is a central extension, then \eqref{1} and \eqref{2} are equivalent to
\item $\uce \colon L \to M$ is a universal central extension of $M$.
In this case,
\begin{enumerate}
\item both $L$ and $M$ are perfect and
\item $\zena L = \uce^{-1}(\zena L) = \zena M$.
\end{enumerate}
\end{enumerate}
\end{Th}

\begin{proof}
By definition, (1) is equivalent to (2). Suppose that (3) holds, we want to proof (a). Let be the product as $K$-Lie algebras
 $L \times L / \conm{L}$. In this case, this is actually a Lie--Rinehart algebra, with the usual operations and action $(x, y + \conm{L})(a) = x(a)$, because
\[
[(x, y + \conm{L}), a(x', y' + \conm{L})] = (a[x, x'], 0) + (x(a)x', 0)
\]
and
\[
a[(x, y + \conm{L}), (x', y' + \conm{L})] = (a[x, x'], 0),
\]
\begin{multline*}
(x, y + \conm{L})(a)(x', y' + \conm{L}) = (x(a)x', x(a)y' + \conm{L})\\
= (x(a)x', [x,ay'] - a[x, y'] + \conm{L}) = (x(a)x', 0).
\end{multline*}

Now we can define the central extension $\bar{\uce} \colon L \times L / \conm{L} \to M$, and two maps $f$ and $g$
\[
    \xymatrix{ L \ar[rr]^\uce \ar@<-1ex>[rd]_g \ar[rd]^f & & M \\
                & L \times L / \conm{L} \ar[ru]_{\bar{\uce}} &}
\]
where $f(x) = (x, x + \conm{L})$ and $g(x) = (x, 0)$. Since $\uce$ is universal, $f$ and $g$ must be equal, so $L/ \conm{L}$ = 0. By the surjectivity of $\uce$, $M$ is perfect too. The assertion (b) is consequence of Lemma \ref{perfect}(b).

We can proof now (3) $\Rightarrow$ (1). Let $f \colon L' \to L$ a central extension. By Lemma \ref{perfect}(c) $\uce f$ is a central extension too, so by the universality of $\uce$, it exists $g \colon L \to L'$ such that $\uce fg = \uce$ and by Lemma \ref{ctrick}(b) $fg = 1_L$.

To show (1) $\Rightarrow$ (3), for a central extension $f \colon N \to M$ we construct as in Lemma \ref{pblemma} the central extension $p_L$. Since $p_L$ splits uniquely, by Lemma \ref{pblemma} it exists a unique map $h \colon L \to N$ such that $fh = \uce$
\end{proof}

\begin{Co}
\label{compo}
Let $f \colon E \to L$ and $g \colon L \to M$ be central extensions. Then $gf \colon E \to M$ is a universal central extension
if and only if $f$ is a universal central extension.
\end{Co}

\begin{proof}
The extension $gf$ is central because $E$ is perfect, so we can apply Lemma \ref{perfect}(c). Hence, $f$ is universal if and only if $1_E \colon E \to E$ is universal, if and only if $gf$ is universal.
\end{proof}

\begin{Co}
Let $L$ and $L'$ be perfect Lie--Rinehart algebras, with universal central extensions $\uce \colon \lle \to L$ and
$\uce' \colon \lle' \to L'$ respectively. Then
\[ L/\zena(L) \cong L'/\zena(L') \ \Longleftrightarrow \ \lle \cong \lle'.\]
\end{Co}

\begin{proof}
Given the diagram
\[
    \xymatrixcolsep{4pc}\xymatrix{ \lle \ar[r]^\uce \ar[d]_\phi & L  \ar[r]^-\pi & L/ \zena(L) \ar[d]^\varphi \\
               \lle' \ar[r]^{\bar{\uce}}       & L' \ar[r]^-{\pi'} & L/ \zena(L'), }
\]
we know that $\phi$ exists and is an isomorphism if and only if $\varphi$ exists and is an isomorphism.
 Since $\pi \uce$ and $\pi'\uce'$ are universal central extensions by Corollary \ref{compo} and $L/ \zena(L)$
  is isomorphic to $L'/ \zena(L')$, by the uniqueness of the universal central extension, $\lle \cong \lle'$.
   Conversely, by Corollary \ref{corzentr} $L/ \zena(L)$ is centreless.
   By Lemma \ref{perfect} (b) $\zena(\lle) = \Ker(\pi\uce)$ and $\zena(\lle') = \Ker(\pi'\uce')$.
    Therefore, $\Ker(\pi'\uce'\phi) = \phi^{-1}\big(\Ker(\pi'\uce')\big) = \phi^{-1}\big(\zena(\lle')\big) = \zena(\lle) = \Ker (\pi\uce)$.
     Since $\pi\uce$ and $\pi'\uce'\phi$ are surjective, $\varphi$ exists and is an isomorphism.
\end{proof}

Note that the results obtained in this section generalize classic results of Lie algebras (see \cite{Gar}).

\subsection{The functor $\ucea$}

Let $L$ be a Lie--Rinehart  $\A$-algebra. We
denote by $M_{\A}L$ the $\A$-submodule of $\A \otimes_K L \otimes_K L$
spanned by the elements of the form
\begin{enumerate}
\item $a \otimes x \otimes x$
\item $a \otimes x \otimes y + a \otimes y \otimes x$
\item $a \otimes x \otimes [y,z] + a \otimes y \otimes [z,x] + a \otimes z \otimes [x,y]$
\item $a \otimes [x,y] \otimes [x',y'] + [x,y](a) \otimes x' \otimes y' - 1 \otimes [x,y] \otimes a[x',y']$
\end{enumerate}
with $x, x', y, y', z \in L$ and $a \in \A$, and put
\[\ucea L :=  \A \otimes_K L \otimes_K L/ M_{\A}L.\]
We shall write $(a,x,y) := a \otimes x \otimes y + M_{\A}L \in \ucea L$.

By construction, the following identities hold in in $\ucea$:
\begin{enumerate}
\item $(a,x,y) = - (a,y,x)$,
\item $(a,x,[y,z]) + (a,y,[z,x]) + (a,z,[x,y]) = 0$,
\item $(1,[x,y],a[x',y']) = (a,[x,y],[x',y']) + ([x,y](a), x', y')$.
\end{enumerate}

The map of $\A$-modules $\A \otimes_K L \otimes_K L \to L$, determined by $(a,x,y) \mapsto a[x,y]$, which vanishes  on
$M_{\A}L$ and hence descends to a linear map
\[\uce \colon \ucea L \to L.\]
Note that
\[ \Ker \uce = \big\{ \sum_i (a_i,x_i,y_i) \ : \ \sum_i a_i[x_i,y_i] = 0 \big\}\]
It is an easy, but tedious calculation to see that the module $\ucea L$ becomes a Lie--Rinehart algebra with
the product
\[[(a,x,y), (a',x',y')] := (aa',[x,y],[x',y']) + (a[x,y](a'), x',y') - ([x',y'](a) a', x,y),\]
and action
\[(a,x,y)(b) := a[x,y](b).\]

It then follows that $\uce \colon \ucea L \to \{L,L\}$ is a central extension of $\{L,L\}$.

Let $f \colon L \to M \in \lrak$. Let $M_\A M \in \A \otimes_K M \otimes_K M$ defined analogously to $M_\A L$. The map
$1_\A \otimes_K f \otimes_K f \colon M_\A L \to M_\A M$ induces an $\A$-linear map
\[\ucea(f) \colon \ucea L \to \ucea M , \quad (a,x,y) \mapsto \big(a,f(x),f(y)\big).\]

Note that the following diagram commutes by construction,

\begin{equation}\label{diaguce}
\xymatrix{
  \ucea L \ar[d]_{\uce_L} \ar[rr]^{\ucea(f)}
               & & \ucea M \ar[d]^{\uce_M}  \\
  L  \ar[rr]_{f}
               & & M             }
\end{equation}

To check that $\ucea f$ is a morphism it suffices to show that
\[
\ucea(f)([(a,x,y),(a',x,'y,)]) = [\ucea(f)(a,x,y),\ucea(f)(a',x',y')],
\]
which since $f$ is a Lie--Rinehart homomorphism, we have that $a[x,y](a') = f(a[x,y])(a') = a[f(x), f(y)](a')$ and the proof follows immediately.

\begin{Pro}
\label{lastprop}
Let $f \colon L \to M$ be a morphism of Lie--Rinehart algebras and suppose that $g \colon M' \to M$ is a central extension.
Then there exists a homomorphism $\mathfrak{f} \colon \ucea L \to M'$, making the following diagram commute,
\begin{equation}
\label{lift}
\xymatrix{
  \ucea L \ar[d]_{\uce} \ar[rr]^{\mathfrak{f}}
               & & M' \ar[d]^{g}  \\
  L  \ar[rr]_{f}
               & & M,             }
\end{equation}
the map $\mathfrak{f}$ is uniquely determined on the derived algebra $\{ \ucea L, \ucea L\}$, by the commutativity
of  \eqref{lift}.
\end{Pro}

\begin{proof}
Let $s \colon M \to M'$ a section of $g$ in Set. The map $s$ may not be linear but we know that $s(km) - ks(m) \in \Ker g \subset \zena(M')$ and $s(m+n) - s(m) - s(n) \in \Ker g \subset \zena(M')$ for $k\in K$ and $m,n \in M$. Using this, we claim that the map
\begin{align*}
\A \times L \times L & \xrightarrow{ \ \ \bar{f}  \ }  M' \\
(a,x,y) & \xmapsto{ \ \ \ \ } a[sf(x), sf(y)],
\end{align*}
is bilinear, since $a[sf(kx), sf(y)] = a[sf(kx)-ksf(x)+ksf(x), sf(y)] = a[ksf(x), sf(y)]$.
The other property follows in the same way.
By the universal property of tensor product, $\bar{f}$ defines a unique map between $\A \otimes_K L \otimes_K L$ and $M'$.
In addition, the map is zero in $M_\A L$, so it can be extended to $\mathfrak{f} \colon \ucea L \to M'$, making the diagram commutative.
 This map conserves the action on $\dka$ because the section $s$ must conserve it too.
 Using the property that $a[x,y](a') = f(a[x,y])(a') = a[f(x), f(y)](a')$, it follows immediately
 that $\mathfrak{f}$ is a Lie algebra homomorphism, hence it is a Lie--Rinehart algebra homomorphism
  that makes the diagram commutative. The uniqueness in $\conm{\ucea L}$ follows from Lemma \ref{ctrick}(b).
\end{proof}

\begin{Th}
Let $L$ be a perfect Lie--Rinehart algebra. Then
\[
\Ker \uce  \xrightarrow { \ \  }   \ucea L  \xrightarrow {\ \uce \ }   L ,
\]
is a universal central extension of $L$. Moreover, if $L$ is centreless, then $\Ker \uce = \zena(\ucea L)$.
\end{Th}

\begin{proof}
It can be seen that $\ucea(\conm{L}) \subset \conm{\ucea L} \subset \ucea L$.
Thus when $L$ is perfect, $\conm{\ucea L} = \ucea L$, so applying Proposition \ref{lastprop}
 for every central extension $f \colon M \to L$ we have a unique map $\mathfrak{f} \colon \ucea L \to M$ making the diagram commutative.
  In other words, $\ucea L$ is the universal central extension of $L$.
\end{proof}
\begin{Rem}
In many algebraic structures as Lie algebras, $\Ker \uce$ is the second homology group with trivial coefficients. However, this is not possible here since we do not have a canonical right $(\A, L)$-module structure in $\A$ as we have seen in Remark~\ref{rem}.
\end{Rem}

\section{Lifting automorphisms and derivations} \label{S:lift}

Let $f \colon L' \to L$ be a covering. Recovering the commutative diagram \eqref{diaguce} we get
\[
\xymatrix{
  \lle' \ar[d]_{\uce'} \ar[rr]^{\mathfrak{f}}
               & & \lle \ar[d]^{\uce}  \\
  L'  \ar[rr]_{f}
               & & L             }
\]
where $ \lle' = \ucea (L')$, $\lle = \ucea (L)$ and $\uce' = \uce_{L'}$, $\uce = \uce_L$.
Since both $\uce'$ and $f$ are central extensions, by Corollary \ref{compo} we know that $f\uce' \colon \lle' \to L$ is a universal central extension of $L$. By the uniqueness of the universal central extension, we know that $\lle' \cong \lle$. In addition, since $\mathfrak{f}$ is a morphism from the universal central extension $f\uce'$ to the central extension $\uce$, it must be an isomorphism. Therefore, we get a covering $\uce' \mathfrak{f}^{-1} \colon \lle \to L'$ with kernel
\[
C := \Ker(\uce'\mathfrak{f}^{-1}) = \mathfrak{f}(\Ker \uce').
\]

\subsection{Lifting of automorphisms}

\begin{Th}\label{th:liftaut}
Using the notation of the beginning of this section,
\begin{itemize}
\item[(a)] Let $h \in {\sf Aut}(L)$. Then there exists $h' \in {\sf Aut}(L')$ such that the diagram
\begin{equation}\label{liftaut}
\xymatrix{
  L' \ar[d]_{h'} \ar[r]^{f}
                & L \ar[d]^{h}  \\
  L' \ar[r]_{f}
                & L             }
\end{equation}
commutes if and only if $\ucea(h)(C) = C$. Moreover, $h'$ is uniquely determined by the diagram \eqref{liftaut} and $h'(\Ker f) = \Ker f$.
\item[(b)] The group homomorphism which sends $h \in {\sf Aut}(L)$ to its lifting $h'\in {\sf Aut}(L')$
\begin{equation*}
\{h \in {\sf Aut}(L) : \ucea(h)(C) = C\} \to \{g \in {\sf Aut}(L') : g(\Ker f) = \Ker f\}
\end{equation*}
is a group isomorphism.
\end{itemize}
\end{Th}

\begin{proof}
(a) If $h'$ exists, it is a morphism from the covering $hf$ to the covering $f$ so by Lemma \ref{ctrick}(b) it is uniquely determined by the commutative of the diagram \eqref{liftaut}. Let suppose then that $h'$ exists. If we apply the $\ucea$ functor to the diagram \eqref{liftaut}, we obtain the commutative diagram
\[
\xymatrix{
  \lle' \ar[d]_{\ucea(h')} \ar[r]^{\mathfrak{f}}   & \lle \ar[d]^{\ucea(h)}  \\
  \lle' \ar[r]_{\mathfrak{f}}                       & \lle        }
\]
In this way, $\ucea(h)(C) = \ucea(h)\big(\mathfrak{f}(\Ker \uce')\big) = (\ucea(h) \: \circ \: \mathfrak{f})(\Ker \uce') = (\mathfrak{f} \: \circ \: \ucea(h'))(\Ker \uce') = \mathfrak{f}(\Ker \uce') = C$.
Suppose now that $\ucea(h)(C) = C$. We obtain the commutative diagram
\begin{equation}\label{liftaut2}
\xymatrix{
\lle \ar[d]_{\ucea(h)} \ar[rr]^{\uce' \mathfrak{f}^{-1}} & & L' \ar@{-->}[d] \ar[rr]^f & & L \ar[d]^{h} \\
\lle                   \ar[rr]^{\uce' \mathfrak{f}^{-1}} & & L'        \ar[rr]^f & & L
}
\end{equation}
If $\ucea(h)(C) = C$, the kernel of the epimorphism $\uce'\mathfrak{f}^{-1} \, \circ \, \ucea(h)$ is $C$, i.e. the kernel of $\uce' \mathfrak{f}^{-1}$. In this way, we obtain an automorphism $h' \colon L' \to L'$ such that \eqref{liftaut2} commutes. The condition that $h'(\Ker f) = \Ker f$ follows immediately by the commutativity of \eqref{liftaut}.

(b) The map is well defined and injective by part (a) of the Theorem. Let $g \in {\sf Aut}(L')$ such that $g(\Ker f) = \Ker f$. It descends to $h \in {\sf Aut}(L)$ such that $fg = hf$. Again by (a), $g$ must be the lifting of $h$ and since the lifting exists it follows that $\ucea(h)(C) = C$.
\end{proof}

\begin{Co}\label{co:liftaut}
If $L$ is perfect, the map
\[
{\sf Aut}(L) \to \{ g \in {\sf Aut}\big(\ucea(L)\big) : g(\Ker \uce) = \Ker \uce \}
\]
that sends $f$ to $\ucea(f)$, is a group isomorphism. Moreover, if $L$ is centreless, ${\sf Aut}(L) \cong {\sf Aut}\big(\ucea(L)\big)$.
\end{Co}

\begin{proof}
Applying the last theorem to the covering $\uce \colon \ucea(L) \to L$, we have that $\uce'$ is the identity map, so $C = 0$ and the corollary follows immediately. By \ref{perfect}(b), if $L$ is perfect we have that $\Ker u = \zena \ucea(L)$ and since every automorphism leaves the centre invariant it is straightforward that ${\sf Aut}(L) \cong {\sf Aut}\big(\ucea(L)\big)$.
\end{proof}

\subsection{Lifting of derivations}

\begin{De}
Let $L$ be a Lie--Rinehart algebra over $\A$. A derivation of $L$ is a pair $D := (\delta, \delta_0)$, where $\delta \colon L \to L$ is a derivation of $L$ as a $K$-Lie algebra, $\delta_0 \in \dka$ and the following identities hold:
\begin{align*}
\delta(ax) &= a\delta(x) + \delta_0(a)x, \\
\delta_0\big(x(a)\big) &= x\big(\delta_0(a)\big) + \delta(x)(a),
\end{align*}
with $a \in \A$ and $x \in L$. Note that the second identity means that the following diagram commutes,
\[
\xymatrix{
L    \ar[d]_{\alpha}           \ar[rr]^\delta & & L \ar[d]^{\alpha} \\
\dka \ar[rr]_{[\delta_0, -]}                  & & \dka    }
\]
We shall write $\derrin(L)$ the $\A$-module of all derivations of the Lie--Rinehart algebra $L$. Observe that $\derrin(L)$, with Lie bracket $[(\delta, \delta_0), (\delta', \delta_0')] = ([\delta, \delta'], [\delta_0, \delta_0'])$ and anchor map $\derrin(L) \to \dka$, $(\delta, \delta_0) \mapsto \delta_0$ is a Lie--Rinehart algebra over $\A$.
\end{De}

Given a derivation $D = (\delta, \delta_0) \in \derrin(L)$, one can define $\ucea(D) = (\delta^{\uce}, \delta_0)$, where $\delta^{\uce}$ is defined in generators as $(a, x, y) \mapsto (\delta_0(a), x, y) + (a, \delta(x), y) + \big(a, x, \delta(y)\big)$. It is a straightforward verification that he map $\ucea(D)$ is also a derivation of the Lie--Rinehart algebra $\ucea(L)$ and yields the following commutative diagram
\[
\xymatrix{
\ucea(L)    \ar[d]_{\uce}           \ar[rr]^{\delta^{\uce}} & & \ucea(L) \ar[d]^{\uce} \\
L \ar[rr]_{\delta}                  & & L    }
\]
leaving $\Ker \uce$ invariant. Moreover, the map
\[
\ucea \colon \derrin(L) \to \{ F = (\gamma, \gamma_0) \in \derrin\big(\ucea(L)\big) : \gamma(\Ker \uce) \subset (\Ker \uce) \},
\]
sending $D$ to $\ucea(D)$ is a Lie--Rinehart homomorphism, and its kernel is contained in the subalgebra of those derivations vanishing on $\{L, L\}$.

We can check the functoriality of $\ucea$ for derivations, in the following lemma.

\begin{Le}
Let $f \colon L \to M$ be a morphism of Lie--Rinehart algebras and let $(\delta_L, \delta_{0}) \in \derrin(L)$ and $(\delta_M, \delta_0) \in \derrin(M)$ be such that $f \delta_L = \delta_M f$. Then, we have that
\[
\ucea(f)\delta^{\uce}_L = \delta_M^{\uce}\ucea(f).
\]
\end{Le}

\begin{proof}
It suffices to check it for an element $(a, x, y) \in \ucea(L)$.
\begin{align*}
\ucea(f)\delta^{\uce}_L(a, x, y) &= \ucea(f)\big((\delta_0(a), x, y) + (a, \delta_L(x), y) + (a, x, \delta_L(y))   \big) \\
{} &= \big(\delta_0(a), f(x), f(y)\big) + \big(a, f\big(\delta_L(x)\big), f(y)\big) + \big(a, f(x), f\big(\delta_L(y)\big)\big) \\
{} &= \big(\delta_0(a), f(x), f(y)\big) + \big(a, \delta_M\big(f(x)\big), f(y)\big) + \big(a, f(x), \delta_M\big(f(y)\big)\big) \\
{} &= \delta_M\big(a, f(x), f(y)\big) \\
{} & = \delta_M\ucea(f)(a, x, y).
\end{align*}
\end{proof}

We will state now the analogue of Theorem \ref{th:liftaut} for derivations.
\begin{Th}
Let $f \colon L' \to L$ be a covering of the Lie--Rinehart algebra $L$ and as before, we denote $C= \ucea(f)(\ker \uce')$.
\begin{itemize}
\item[(a)] A derivation $(\delta, \delta_0) \in \derrin(L)$ lifts to a derivation $(\delta', \delta_0)$ of $L'$ satisfying $\delta'f = f\delta$ if and only if the derivation $\delta^{\uce}(C) \subset C$. Moreover, $\delta'$ is uniquely determined and leaves $\Ker f$ invariant.
\item[(b)] The map which sends $(\delta, \delta_0)\in \derrin(L)$ to its lifting $(\delta', \delta_0)\in \derrin(L')$
\[
\{ (\delta, \delta_0) \in \derrin(L) : \ucea(\delta^{\uce})(C) \subset C \} \to \{ (\eta, \eta_0) \in \derrin(L') : \eta^{\uce}(\Ker f) \subset \Ker f \}
\]
is an isomorphism of Lie--Rinehart algebras.
\item[(c)] In particular, from the covering $\uce \colon \ucea(L) \to L$ we obtain that the map
\[
\ucea \colon \derrin(L) \to \{ (\gamma, \gamma_0) \in \derrin\big(\ucea(L)\big) : \gamma(\Ker \uce) \subset \Ker \uce \}
\]
is an isomorphism. Moreover, if $L$ is centreless, we have that $\derrin(L) \cong \derrin\big(\ucea(L)\big)$.
\end{itemize}
\end{Th}

\begin{proof}
The proof is analogue of the proofs of Theorem \ref{th:liftaut} and Corollary \ref{co:liftaut}. The fact that a derivation is not a Lie--Rinehart morphism, do not add any complication.
\end{proof}

\subsection{Universal central extensions of split exact sequences}

\begin{Th}
Let $\xymatrix@C=0.6cm{L \ar[rr]^f & & M \ar@{>>}[rr]_{g} & & N \ar@/_/[ll]_{s}}$ be a split short exact sequence of perfect Lie--Rinehart algebras. We have the following commutative diagram
\[
\xymatrix{
 \ucea(L) \ar[rr]^\varphi \ar[d]_{\uce_L} & & \ucea(M) \ar[rr]_{\gamma} \ar[d]_{\uce_M} & & \ucea(N) \ar@/_/[ll]_{\sigma} \ar[d]^{\uce_N}    \\
 L        \ar[rr]^f                       & & M        \ar[rr]_{g} & &N \ar@/_/[ll]_{s}
}
\]
where $\ucea(M)$ is a semidirect product
\[
\ucea(M) = \varphi\big(\ucea(L)\big) \rtimes \sigma\big(\ucea(N)\big),
\]
and
\[
\Ker \uce_M = \varphi(\Ker \uce_L) \oplus \sigma(\ker_N).
\]
We know that $M \cong L \rtimes N$ since the bottom row exact sequence splits. If $M = L \times N$ is a direct product, i.e. $[f(L), s(N)] = \{0\}$, we have
\[
\ucea(L \times N) \cong \ucea(L) \times \ucea(N).
\]
\end{Th}

\begin{proof}
In order to simplify the notation, we can interpret $f$ and $s$ as identifications, so we will write $l$ for $f(l)$ and $n$ for $s(n)$. Given any $(a, \tilde{n}, \tilde{l}) \in \ucea(M)$ where $\tilde{n} \in N$ and $\tilde{l} \in L$, by the perfectness of $L$ and the properties of $\ucea(M)$
\begin{align*}
(a, \tilde{n}, \tilde{l}) &= (a, b[n, n'], c[l, l']) \\
{} &= (ac, b[n, n'], [l, l']) + (ab[n,n'](c), l, l') \\
{} &= (ac, [b[n, n'], l], l') + (ac, l, [b[n, n'], l]) + (ab[n,n'](c), l, l'),
\end{align*}
which means that $(A, N, L) \subset (A, L, L)$, so $\ucea(M) = (A, L, L) + (A, N, N)$. By definition, $(A, L, L) = \varphi\big(\ucea(L)\big)$ and $(A, N, N) = \sigma\big(\ucea(N)\big)$. Now since $\gamma \sigma$ is the identity map, we know that $\ucea(M) \cong \Ker \gamma \rtimes \sigma\big(\ucea(N)\big)$. In this way, $\sigma\big(\ucea(N)\big) \cong \ucea(N)$ and since $(A, L, L) \subset \Ker \gamma$ it follows that $\Ker \gamma = (A, L, L) = \varphi\big(\ucea(L)\big)$, so we have that $\ucea(M) = \varphi\big(\ucea(L)\big) \rtimes \sigma\big(\ucea(N)\big)$.

Every element of $\ucea(M)$ has the form $\varphi(l) \otimes \sigma(n)$ where $l \in \ucea(L)$ and $n \in \ucea(N)$. This means that any element of $\ucea(M)$ is in $\Ker \uce_M$ if and only if $0 = \uce_M\varphi(l) = \uce_L(l)$ and $0 = \uce_M\sigma(n) = \uce_N(n)$, so $\Ker \uce_M = \varphi(\Ker \uce_L) \oplus \sigma(\ker_N)$.

In the particular case that $M = L \times N$, we can define the induced map
\[
\varphi \times \sigma \colon \ucea(L) \times \ucea(N) \to \ucea(M),
\]
and it is an easy computation that it is a Lie--Rinehart algebra morphism. Moreover, $\Ker (\varphi \times \sigma) = \Ker \varphi$. Given $l \in \Ker \varphi$, $\uce_M\varphi(l) = 0 = \uce_L(l)$ so $l \in \Ker \uce_L \in \zena(L)$, which means that $\varphi \times \sigma$ is a covering. We can use now Theorem \ref{T:ucemain}(2) to see that $\varphi \times \sigma$ is an isomorphism completing the proof.
\end{proof}

\section{Non-abelian tensor product of Lie--Rinehart algebras}
\label{S:tensor}
A non-abelian tensor product of Lie algebras was introduced by Ellis \cite{Ell1}. Here
we adapt some of his results to the case of Lie--Rinehart algebras, in order to use them to obtain
a description of universal central extensions in this category.

Let $L, M \in \lrak$. By an \emph{action} of $L$ on $M$, we mean an $K$-linear map,
$L \times M \to M$, $(x,m) \mapsto \ ^xm$, satisfying
\begin{enumerate}
\item $ ^{x}(am) = a(^{x}m) + x(a)m, $
\item $ ^{[x,y]}m = \ ^{x}(^{y}m) - \ ^{y}(^{x}m), $
\item $ ^{x}[m,n] = [ ^{x}m, n] + [m, ^{x}n ], $
\end{enumerate}
for all $a \in \A$, $x,y \in L$ and $m,n \in M$. For example, if $L$ is a subalgebra of some Lie--Rinehart algebra $\lle$ and
$M$ is an ideal of $\lle$ then the bracket in $\lle$ yields an action of $L$ on $M$.

If we have an action of $L$ on $M$ and an action of $M$ on $L$, for any Lie--Rinehart algebra $\lle$ we call a $K$-bilinear function
 $f \colon L \times M \to \lle$ a \emph{Lie--Rinehart pairing} if
\begin{enumerate}
\item $\alpha_{\lle}\big(f(x, m)\big) = [\alpha_L(x), \alpha_M(m)]$,
\item $f([x, y], m) = f(x, \mbox{}^ym) - f(y, \mbox{}^xm)$,
\item $f(x, [m, n]) = f(^nx, m) - f(^mx, n)$,
\item $\begin{aligned}[t]
    f\big(a(^mx), b(^yn)\big) & =-ab[f(x, m), f(y, n)] - a[\alpha_L(x), \alpha_M(m)](b)f(y, n)\\
   {} &+ [\alpha_L(y), \alpha_M(n)](a)bf(x,m),
  \end{aligned}$
\end{enumerate}
for all $a,b \in \A$, $x,y \in L$ and $m,n \in M$.

We say that a Lie pairing $f \colon L \times M \to \lle$ is \emph{universal} if for any other Lie--Rinehart pairing
 $g \colon L \times M \to \lle'$ there is a unique Lie--Rinehart homomorphism $\varphi \colon \lle \to \lle'$ making commutative the diagram:
\[
\xymatrix{ L \times M \ar[rr]^-f \ar[rrd]_-g & & \lle \ar[d]^\varphi  \\
                                          & &\lle'}
\]
The Lie--Rinehart algebra $\lle$ is unique up to isomorphism which we will describe as the non-abelian tensor product of $L$ and $M$.

\begin{De}\label{D:tp}
Let $L$ and $M$ be a pair of Lie--Rinehart algebras together with an action of $L$ on $M$ and an action of $M$ on $L$.
 We define the \emph{non-abelian tensor product of $L$ and $M$ in ${\normalfont \lrak}$, $L \otimes M$}, as the Lie--Rinehart $\A$-algebra
  spanned as an $\A$-module by the symbols $x \otimes m$, and subject only to the relations:
\begin{enumerate}
\item $k(x \otimes m) = kx \otimes m = x \otimes km$,
\item $x \otimes (m + n) = x \otimes m + x \otimes n$, \\
	  $(x + y) \otimes m = x \otimes m + y \otimes m$,
\item $[x,y] \otimes m = x \otimes \mbox{}^{y}m - y \otimes \mbox{}^{x}m$, \\
	  $x \otimes [m, n] = \mbox{}^{n}x \otimes m - \mbox{}^{m}x \otimes n$,
\item $[a(x \otimes m), b(y \otimes n)] = -ab(^{m}x \otimes \mbox{}^{y}n) + a \alpha(x \otimes m)(b)(y \otimes n) - \alpha(y \otimes n)(a)b(x \otimes m)$,
\end{enumerate}
for every $k \in K$, $a, b \in \A$, $x,y \in L$ and $m, n \in M$. Here the map $\alpha \colon L \otimes M \to \dka$ is given
 by $\alpha\big(a(x \otimes m)\big):=a[\alpha_L(x), \alpha_M(m)]$.
\end{De}

This way, the map $f \colon L \times M \to L \otimes M$ which sends $(x, m)$ to $x \otimes m$ is a universal
 Lie--Rinehart pairing by construction.

\begin{De}
Two actions $L \times M \to M$ and $M \times L \to L$ are said to be \emph{compatible}
if for all $x,y \in L$ and $m,n \in M$,
\begin{enumerate}
\item $- \alpha_L(^{m}x) = \alpha_M(^{x}m) = [\alpha_L(x), \alpha_M(m)]$,
\item $^{(^{m}x)}n = [n, \mbox{}^{x}m]$,
\item $^{(^{x}m)}y = [y, \mbox{}^{m}x]$.
\end{enumerate}
\end{De}

This is the case, for example, if $L$ and $M$ are both ideals of some Lie--Rinehart algebra
and the actions are given by multiplication. We can see another example of compatible actions when
 $\partial \colon L \to N$ and $\partial' \colon M \to N$ are crossed modules. In this case, $L$ and $M$ act on each other via the action of $N$.
 These actions are compatible.

From this point on we shall assume that all actions are compatible.

\begin{Pro}
Let $\mu \colon  L \otimes M \to L$ and $\nu \colon  L \otimes M \to M$ be the homomorphisms defined on generators by
 $\mu\big(a(x \otimes m)\big) = -a(^{m}x)$ and $\nu\big(a(x \otimes m)\big) = a(^{x}m)$. They are Lie--Rinehart homomorphisms and the following diagram is commutative:
\[
\xymatrix{
L \otimes M \ar[rr]^\nu \ar[d]_\mu \ar[rrd]^\alpha & & M \ar[d]^{\alpha_M}  \\
L \ar[rr]_{\alpha_L} & & {\normalfont \dka}}
\]
\end{Pro}

We can relate the Lie--Rinehart tensor product $L \otimes M$ with the tensor product of $L$ and $M$ as an $\A$-module.
 We will denote it by $L\underset{\modu}\otimes  M$ the $K$-module and $\A$-module generated by the symbols $x \otimes m$ subject to the relations
\begin{enumerate}
\item $k(x \otimes m) = kx \otimes m = x \otimes km$,
\item $x \otimes (m + n) = x \otimes m + x \otimes n$, \\
	  $(x + y) \otimes m = x \otimes m + y \otimes m$,
\end{enumerate}
for every $k \in K$, $x,y \in L$ and $m, n \in M$.

\begin{Pro}
The canonical map $L\underset{\modu}\otimes  M \to L \otimes M$ is a $\A$-module homomorphism and is surjective.
 In addition, if $L$ and $M$ act trivially on each other, there is an isomorphism of $\A$-modules:
\[
L \otimes M \cong L^{\ab} \underset{\modu}\otimes   M^{\ab}.
\]
\end{Pro}

\begin{proof}
If $L$ acts trivially on $M$ we have that $x(a)m = 0$ for $a \in \A, x \in L$ and $m \in M$. This means that
\[
a[x,y] \otimes m = x \otimes a(^ym) + x \otimes y(a)m - ay \otimes \mbox{}^xm - x(a)y \otimes m = 0
\]
being straightforward the isomorphism.
\end{proof}

\begin{Pro}
The Lie--Rinehart algebras $L \otimes M$ and $M \otimes L$ are isomorphic.
\end{Pro}

\begin{proof}
The map $f \colon  L \times M \to M \otimes L$ which sends $(x, m) \to m \otimes x$ is a Lie--Rinehart pairing, then by the
 universal property of $L \otimes M$ there is a Lie--Rinehart homomorphism $L \otimes M \to M \otimes L$. In a similar way, we can
construct the inverse $M \otimes L  \to L \otimes M$ and establish an isomorphism.
\end{proof}

\begin{Pro}
Consider the following short exact sequence of Lie--Rinehart algebras
\[
\xymatrix{
 L \ar[r]^f & M \ar[r]^g & N ,}
\]
and assume that $P$ is a Lie--Rinehart algebra which acts compatibly on $L$, $M$ and $N$, and the Lie--Rinehart algebras
 $L, M, N$  act compatibly on $P$. Suppose also that the Lie--Rinehart morphisms $f$ and $g$ conserve these actions, i.e.,
  $f(^pm) = \mbox{}^pf(m)$ and $\mbox{}^mp = \mbox{}^{f(m)}p$. Then, the following sequence is exact
\[
\xymatrix{
L \otimes P \ar[r]^-{f \otimes 1} & M\otimes P \ar[r]^-{g \otimes 1} & N\otimes P .}
\]
\end{Pro}

\begin{proof}
Since $f$ and $g$ conserve the actions, it is easy to see that $f \otimes 1$ and $g \otimes 1$ are Lie--Rinehart algebra morphisms.
 Furthermore, the morphism $g \otimes 1$ is clearly surjective, and $\Im (f\otimes 1) \subset \Ker (g\otimes 1)$.
  Since $fg = 0$, we have that $f(x)(a) = 0$ for every $a \in \A$ and $x \in L$.
   This means that $(f \otimes 1)(x\otimes p)(a) = [\alpha_M\big(f(x)\big), \alpha_P(p)](a) = 0$.
   Moreover, $\Im (f\otimes 1)$ is an $\A$-module and conserves the Lie bracket since $f$ and $g$ conserve the actions, so $\Im (f\otimes 1)$ is an ideal.
Then to proof the other inclusion, we will show that $M\otimes P / \Im (f\otimes 1) \cong N \otimes P$.
Since $\Im (f\otimes 1) \subset \Ker (g\otimes 1)$ we have a natural epimorphism $\phi \colon  M\otimes P / \Im (f\otimes 1) \to N \otimes P$.
 Now we define the map $\varphi \colon  N \times P \to M \otimes P / \Im(f\otimes 1)$ such that $\varphi(n,p) = m \otimes p + \Im(f\otimes 1)$
 where $m$ is such that $f(m) = n$. It follows that it is a Lie pairing, so by the universality of the tensor product,
  there exists a unique Lie--Rinehart morphism $\bar{\varphi} \colon  N\otimes P \to M \otimes P / \Im(f\otimes 1)$,
   and it is straightforward that $\phi$ and $\varphi$ are inverse morphisms.
\end{proof}

\begin{Pro}
Given a perfect Lie--Rinehart algebra $L$, the non-abelian tensor product $L \otimes L$ where the action of $L$ on $L$ is the Lie bracket, with the additional relation
\[
(a[x, y] \otimes b[x', y']) = ab([x,y] \otimes [x', y']) - b[x',y'](a)(x \otimes y) + a[x,y](b)(x' \otimes y'),
\]
denoted by $L {\hat{\otimes}} L$, is the universal central extension of $L$.
\end{Pro}

\begin{proof}
It is routine to check that $L {\hat{\otimes}} L \to L$ is a central extension. To see the universality, given a central extension $p: M \to L$,
 we pick a section in Set $s: L \to M$. We define now a map $f: L \times L \to M$ by $f(x, y) = [s(x), s(y)]$.
  Doing the same trick as in Proposition \ref{lastprop}, we see that is a Lie--Rinehart pairing,
   so it can be extended to $L \otimes L \to M$. It is easy to see that the map vanishes in the elements of the additional relation.
    Since $L$ is perfect, we saw in Lemma \ref{ctrick} that this map is unique.
\end{proof}


\begin{thebibliography}{99}

\bibitem{CaEi}
H.~Cartan and S.~Eilenberg.
\newblock {\em Homological algebra}.
\newblock Princeton University Press, Princeton, N. J., 1956.

\bibitem{CLP1}
J.~M. Casas, M.~Ladra, and T.~Pirashvili.
\newblock Crossed modules for {L}ie-{R}inehart algebras.
\newblock {\em J. Algebra}, 274(1):192--201, 2004.

\bibitem{CLP2}
J.~M. Casas, M.~Ladra, and T.~Pirashvili.
\newblock Triple cohomology of {L}ie-{R}inehart algebras and the canonical
  class of associative algebras.
\newblock {\em J. Algebra}, 291(1):144--163, 2005.

\bibitem{Ell1}
G.~J. Ellis.
\newblock A nonabelian tensor product of {L}ie algebras.
\newblock {\em Glasgow Math. J.}, 33(1):101--120, 1991.

\bibitem{Gar}
H.~Garland.
\newblock The arithmetic theory of loop groups.
\newblock {\em Inst. Hautes \'Etudes Sci. Publ. Math.}, (52):5--136, 1980.

\bibitem{Hue}
J.~Huebschmann.
\newblock Poisson cohomology and quantization.
\newblock {\em J. Reine Angew. Math.}, 408:57--113, 1990.

\bibitem{Hue2}
J.~Huebschmann.
\newblock Lie-{R}inehart algebras, {G}erstenhaber algebras and
  {B}atalin-{V}ilkovisky algebras.
\newblock {\em Ann. Inst. Fourier (Grenoble)}, 48(2):425--440, 1998.

\bibitem{Hue1}
J.~Huebschmann.
\newblock Duality for {L}ie-{R}inehart algebras and the modular class.
\newblock {\em J. Reine Angew. Math.}, 510:103--159, 1999.

\bibitem{Kal}
R.~K{\"a}llstr{\"o}m.
\newblock Smooth modules over {L}ie algebroids {I}.
\newblock arXiv:math/9808108v1, 1998.

\bibitem{Kap}
Mikhail Kapranov.
\newblock Free {L}ie algebroids and the space of paths.
\newblock {\em Selecta Math. (N.S.)}, 13(2):277--319, 2007.

\bibitem{KrRo}
U.~Kr{\"a}hmer and A.~Rovi.
\newblock A {L}ie-{R}inehart algebra with no antipode.
\newblock arXiv:1308.6770, 2013.

\bibitem{Mac}
K.~Mackenzie.
\newblock {\em Lie groupoids and {L}ie algebroids in differential geometry},
  volume 124 of {\em London Mathematical Society Lecture Note Series}.
\newblock Cambridge University Press, Cambridge, 1987.

\bibitem{Neh}
E.~Neher.
\newblock An introduction to universal central extensions of {L}ie
  superalgebras.
\newblock In {\em Groups, rings, {L}ie and {H}opf algebras ({S}t. {J}ohn's,
  {NF}, 2001)}, volume 555 of {\em Math. Appl.}, pages 141--166. Kluwer Acad.
  Publ., Dordrecht, 2003.

\bibitem{Pal}
R.~S. Palais.
\newblock The cohomology of {L}ie rings.
\newblock In {\em Proc. {S}ympos. {P}ure {M}ath., {V}ol. {III}}, pages
  130--137. American Mathematical Society, Providence, R.I., 1961.

\bibitem{Rin}
G.~S. Rinehart.
\newblock Differential forms on general commutative algebras.
\newblock {\em Trans. Amer. Math. Soc.}, 108:195--222, 1963.

\end{thebibliography}

\end{document}